\documentclass[10pt]{amsart}
\usepackage[all]{xy} 
\usepackage{verbatim}
\usepackage{tikz, amsmath}

\newtheorem{thm}{Theorem}[section]
\newtheorem{cor}[thm]{Corollary}
\newtheorem{prop}[thm]{Proposition}
\newtheorem{lem}[thm]{Lemma}

\theoremstyle{definition}
\newtheorem{defn}[thm]{Definition}

\newtheorem{rmk}[thm]{Remark}

\newtheorem*{ack}{Acknowledgments}

\numberwithin{equation}{section}

\newcommand{\ot}{\otimes}

\newcommand{\Tr}{\mathrm{Tr}}

\newcommand{\C}{\mathbb C}
\newcommand{\Z}{\mathbb {Z}}

\newcommand{\Om}{\Omega}

\title{A loop group extension of the odd Chern character}

\author{Scott O. Wilson}
  \address{Scott O. Wilson, Department of Mathematics, Queens College, City University of New York, 65-30 Kissena Blvd., Flushing, NY 11367}
  \email{scott.wilson@qc.cuny.edu}

\begin{document}
\maketitle

\begin{abstract}
We show that the universal odd Chern form, defined on the stable unitary group $U$, extends to the loop group $LU$ 
in a way that is closed with respect to an equivariant-type differential.
This provides an odd analogue to the Bismut-Chern form. We also describe the associated transgression form,
 the so-called Bismut-Chern-Simons form, and explicate some properties it inherits as a differential form on the space of maps of a cylinder into the stable unitary group. As a corollary, we obtain the Chern character homomorphism from odd K-theory to the periodic cohomology of the free loop space, represented geometrically on the level of differential forms.
\end{abstract}
\allowdisplaybreaks

\setcounter{tocdepth}{2}
\tableofcontents

\section{Introduction}

The Chern character homomorphism plays a fundamental role in  topology and geometry as it relates K-theory to ordinary cohomology. 
Several recent works have emphasized differential refinements of cohomology theories \cite{HS}, \cite{BS}, and in \cite{TWZ3}, the authors
show that a careful study of the odd Chern Character, represented geometrically by an odd differential form on the stable unitary group $U$, leads to a rather explicit and elegant differential refinement of odd K-theory.

On the other hand, the even degree Chern form, associated to a connection $\nabla$ on a bundle over $M$, has been shown by Bismut to extend to a differential form on the free loop space $LM$ that is closed with respect to an equivariant-type differential \cite{B}.  
Bismut introduced this differential form as a contribution to the integrand of a path integral over the free loop space which calculates (non-rigorously) the index of a certain twisted Dirac operator.  This work was subsequently reinterpreted in terms of cyclic homology by \cite{GJP}, with further generalizations and applications in several works. In  \cite{TWZ} the authors provide an analogous result for abelian gerbes, whereas in \cite{TWZ2} the relationship to 
even K-theory is illuminated. Finally, in \cite{Ha} and \cite{ST}, the Bismut-Chern form has been given a rigorous field theoretic interpretation, appearing naturally in the  dimensional reduction from a $1|1$ supersymmetric Euclidean field theory on M, to a $0|1$  field theory on $LM$.

This leads to the natural question of whether the odd Chern form on $U$ extends to a differential form on the loop group $LU$ in a way that is  closed with respect to the equivariant-type differential. We answer this question affirmatively by constructing this so-called 
odd Bismut-Chern form using transgression techniques we establish below. Interestingly this differential form on $LU$ is given by an explicit iterated integral  formula involving only the left invariant $1$-form on $U$ and its directional derivatives. We are optimistic that this will have a field theoretic interpretation as well.

The contents of this paper are as follows. In the next section we provide preliminary definitions and results concerning the complex of differential forms on free loopspaces that is used in the paper. In section \ref{sec;BCh} we provide background on 
the Bismut-Chern form, as well as a re-interpretation as a universal form on the free loopspace $LBU$. 
The next section provides a general method for constructing transgression forms on free loopspaces, and defines the transgression
form associated to the Bismut-Chern form, first introduced in \cite{TWZ2}. We call this the Bismut-Chern-Simons form 
since it restricts along constant loops to the Chern-Simons form.

In section \ref{sec;oddCh} we provide a brief review of the odd Chern form on $U$ and in the subsequent sections define the 
odd Bismut-Chern form and prove several fundamental properties. In particular, we give an explicit formula for the 
transgression of this form, which is an even differential form on the space of maps of cylinder into $U$,
whose value in degree zero is a lift of the ``winding number'' function. We close with an application which 
produces a lift of the the odd Chern Character homomorphism $Ch:K^{-1} \to H^{odd}$ to a natural transformation with image in the  
the completed periodic cohomology of the free loopspace.

\begin{ack}
I would like to thank Thomas Tradler and Mahmoud Zeinalian for helpful discussions concerning the topics of this paper. The author was supported in part by The City University of New York PSC-CUNY Research Award Program.
\end{ack}

\section{Preliminaries}
For a smooth manifold $M$, let $LM$ denote the space of smooth loops in $M$, considered as a diffeological (Chen) space \cite{C1}, \cite{C2}, or as a Fr\'echet space, \cite{H}. In fact, these structures agree, \cite{L}, \cite{Wa} Lemma A.1.7.

 The free loop space $LM$ has a natural vector field, given by the circle action, whose induced contraction operator on differential forms is denoted by $\iota$. We use the notation $\Om_{d + \iota}^j$ to denote the vector space of invariant complex valued differential forms of degree $j$, i.e. those that are in the kernel of $d \iota + \iota d$.

Let 
\[
\Omega_{d + \iota}^{even}(LM) =  \prod_{k \geq 0} \Omega^{\textrm{2k}}_{d+ \iota}(LM) ,
 \quad 
 \Omega_{d + \iota}^{odd}(LM) =  \prod_{k \geq 0} \Omega^{\textrm{2k+1}}_{d+ \iota}(LM). 
\]
The operator $d + \iota$ defines a $\Z_2$-graded complex and we denote the cohomology by 
 \[
 H^{\textrm{even}}_{d + \iota}(LM), 
 \quad H^{\textrm{odd}}_{d+ \iota}(LM).
 \]
The earliest reference for these groups are  \cite{W}, \cite{A}, and \cite{B},  with more recent related work in \cite{KM}.

 In \cite{JP} the authors show these cohomology groups are isomorphic to the \emph{completed periodic equivariant cohomology}
 $h^*_{S^1}(LM)$ in even and odd degrees, respectively. Recall this is the cohomology of the \emph{completed periodic equivariant complex} which is defined to be 
the differential graded ring $(\Omega(LM)[[u,u^{-1}]], d+ u \iota)$ consisting of formal power series in $u, u^{-1}$ with coefficients in 
$\Om_{d+ \iota}(LM)$, where $u$ is an indeterminant of degree $2$. Note this allows for elements with arbitrary high powers of $u^{-1}$, but
not $u$, since forms are concentrated in positive degrees.
All of the results  in this paper can be restated in terms in the completed periodic equivariant complex by introduction of this formal variables $u^{-1}$, and the complex with differential $d + \iota$ is recovered by setting $u=1$. 
 
Finally, it  follows from Theorem 2.1 of \cite{JP} that the inclusion of constant loops $\rho: M \to LM$ induces an isomorphism
  \[
\rho^*:  h^*_{S^1}(LM) \to u^{-1}H^*_{S^1}(M) = H^*(M) \otimes \C[u,u^{-1}]
 \]
 where $u^{-1}H^*_{S^1}(M)$ is the localization with respect to $u$ of the usual equivariant cohomology of $M$ (with the trivial circle action). 
 In summary, we have (for each $n \geq 0$) that the restriction to constant loops induces an isomorphism
\begin{eqnarray*}
H^{even}_{d + \iota}(LM)  \cong  \left( H^*(M) \otimes \C[u,u^{-1}] \right)_{2n} \cong  H^{even}(M) \\ 
H^{odd}_{d+ \iota}(LM)  \cong \left( H^*(M) \otimes \C[u,u^{-1}] \right)_{2n+1} \cong  H^{odd}(M). 
\end{eqnarray*}
 So we obtain an expression of this equivariant-type cohomology of the free loopspace in terms of the ordinary cohomology of $M$.

\section{The even Bismut-Chern form} \label{sec;BCh}

In \cite{B} the author showed that the even Chern form of a connection on a bundle extends to an 
$(d + \iota)$-closed even differential form on the free loopspace of the base. This  works was revisited in \cite{GJP}, in the context of cyclic homology.  In \cite{TWZ}, and \cite{TWZ2}, the authors re-interpreted the so-called Bismut-Chern as a $(d+ \iota)$-closed extension of the trace of holonomy function on the loopspace, and developed several properties with respect to sum and tensor products of connections.   We review this material in this section.

Let $E \to M$ be a complex vector bundle with connection $\nabla$. On any  single local trivialization $V$ of $M$, we can write the connection locally as a matrix $A$ of 
$1$-forms, with curvature $R$. There is a degree $2k$ form $Tr(hol^V_{2k})$  on $LV$ given by
\begin{equation} \label{eq:BCh^V_{2k}}
Tr (hol^V_{2k}) = Tr \left(\sum_{m \geq k} \, \,  \sum_{1 \leq j_1 < \dots < j_k \leq m} \int_{\Delta^m} X_1(t_1) \cdots X_m(t_m) dt_1 \cdots dt_m \right) ,
\end{equation}
where
\[
X_j (t_j) = \left\{
\begin{array}{rl}
R(t_j) & \text{if  } j \in \{ j_1, \dots , j_k\} \\
\iota A(t_j)  & \text{otherwise}
\end{array} \right.
\]
Here $R(t_j)$ is a $2$-form taking in two vectors at $\gamma(t_j)$ on a loop $\gamma \in V$ , and 
$\iota A(t_j) = A (\gamma'(t_j))$. Note that $Tr(hol_0)$ is simply the trace of the 
usual holonomy, and heustically $Tr(hol^V_{2k}) $ is given by the same formula for the trace of holonomy, except with exactly $k$ copies of the function $\iota A$ replaced by the $2$-form $R$, summed over all possible places.

More generally, a global form is defined on $LM$ is defined using a open covering of $LM$ induced by an open covering of $M$ over which the bundle is locally trivialized. The formula above, essentially glued together by transition data, is shown to define an even differential form on $LM$, independent of choice of local data, which we denote by $Tr(hol_{2k})$, \cite{TWZ}.
We define
\[
BCh(\nabla)= \sum_{k \geq 0} Tr(hol_{2k}).
\]

\begin{prop} (\cite{B}) \label{prop:BChproperties}
For any connection $\nabla$ on a complex vector bundle $E \to M$, we have
\[
(d+ \iota) BCh(\nabla) = 0.
\]
 Therefore, $BCh(\nabla) $ determines a class $[BCh(\nabla)] \in \Om^{even}_{S^1}(LM)$. Moreover, we have
\[
\rho^* BCh(\nabla) = Ch(\nabla)
\]
where $\rho^* :\Om^{even}_{S^1} (LM) \to \Om^{even}(M)$ is the restriction to constant loops and $Ch(\nabla) = Tr\left( e^ R  \right) 
\in \Om^{even}(M)$ is the ordinary even Chern character.
\end{prop}

We have omitted the scalar factor $1/2 \pi i$ that is sometimes used in the definition of the even Chern character for its
relation to integral cohomology. By replacing  $R$ by $R/2 \pi i$ in the definition of the Bismut-Chern form, we obtain 
a form that restricts to $Tr\left( e^{R/ 2 \pi i}  \right)$, and is closed with respect to $d + (2 \pi i) \iota$. 
 This operator produces a complex that is isomorphic to the one used here.
In \cite{TWZ2} the following properties are proved.

\begin{thm} (\cite{TWZ2}, Theorem 3.3) \label{thm:BChsumtensor}
Let $(E,\nabla) \to M$ and $(\bar E, \bar \nabla) \to M$ be complex vector bundles with connections. 
Let $\nabla \oplus \bar \nabla$ be the induced connection on $E \oplus \bar E \to M$, and 
$\nabla \otimes \bar \nabla := \nabla \otimes Id + Id  \otimes \bar \nabla$ be the induced connection on $E \otimes \bar E \to M$. Then
\[
BCh(\nabla \oplus \bar \nabla) = BCh(\nabla) + BCh(\bar \nabla)
\]
and
\[
BCh(\nabla \ot \bar \nabla) = BCh(\nabla) \wedge BCh(\bar \nabla).
\]
\end{thm}

\begin{rmk} \label{rmk;BChUniv}
 A complex vector bundle $E \to M$ of rank $k$ with hermitian connection $\nabla$ can always be given by the pullback along a map $M \to BU(k)$ of the universal bundle with universal connection over $BU(k)$, \cite{NR}.  
Here our model for $BU(k)$ is the limit of the Grassmann manifolds of $k$-planes in $\C^n$, with tautological bundles over them given by vectors in the $k$-planes. These bundles sit inside the trivial bundles with trivial connection given by exterior $d$, and the tautological sub-bundle has an induced
connection given by the composition of projection $P$ onto the $k$-plane composed with $d$, i.e. $\nabla = P \circ d$.
The curvature of this connection is given by $R = P (dP)^2$. In this way we can describe the
universal Bismut-Chern form on $L BU(k)$. An explicit formula is given by 
\begin{equation}
Tr \left(  \sum_{\stackrel{2k \geq 0}{m \geq k}} \, \,  \sum_{1 \leq j_1 < \dots < j_{2k} \leq m} \int_{\Delta^m} X_1(t_1) \cdots X_m(t_m) dt_1 \cdots dt_m \right) ,
\end{equation}
where
\[
X_j (t_j) = \left\{
\begin{array}{rl}
P(dP)^2(t_j) & \text{if  } j \in \{ j_1, \dots , j_{2k}\} \\
\iota P(dP)(t_j)  & \text{otherwise}
\end{array} \right.
\]

\end{rmk}

\section{Transgression forms on free loop spaces}
Assigning a differential form to some geometric data on a manifold often applies as well to a compact one-parameter family of the data. By integration along the family, we obtain a form of degree one lower whose exterior derivative is the 
difference between the differential forms assigned to the endpoint data. We explain now that this can also be accomplished for 
the free loopspace $LM$ and the operator $d + \iota$.

Let $I=[0,1]$. There is a map $j: L M \times I \to L(M \times I) = LM \times LI$ which is induced by the inclusion $I \to LI$ of 
constant loops. Note that $j^* : \Om^*(L(M \times I)) \to \Om^*(LM \times I)$ commutes with the contraction operators $\iota$ and exterior derivative $d$ so there is a well defined map $j^* : \Om^*_{d+\iota}(L(M \times I)) \to \Om^*_{d+\iota}(LM)$.

Consider the composition
\[
\xymatrix{
\Om^*(LM \times I)  \ar[d]_-{\int_I} & \Om^*(L(M \times I)) \ar[l]_-{j^*}\\
\Om^*(LM) & 
}
\]
where $\int_I$ is integration along the fiber $I$.  By Stokes' Theorem for integration along fiber, and the fact that for all forms $\omega$ we have
\[
\iota \int_I j^* \omega = \int_I \iota j^* \omega = \int_I j^* \iota \omega,
\]
it follows that we have the following \emph{equivariant Stokes' formula} for integration along the fiber 
\[
(d + \iota) \int_I j^*\omega = \int_{\partial I} j^* \omega + \int_I j^* (d+ \iota) \omega.
\]

\begin{defn} \label{defn:BCS}
Let $\nabla_s$ be a path of connections on a bundle $E \to M$, regarded as a bundle with connection over $M \times I$.
We define the Bismut-Chern-Simons form $BCS(\nabla_s) \in \Om^{odd}_{d + \iota}(LM)$ by
\[
BCS(\nabla_s) = \int_I j^* BCh(\nabla_s) .
\]
\end{defn}

From the discussion above, and the fact that $BCh$ is $(d+\iota)$-closed, we have  
\begin{equation} \label{eq:BCSStokes}
(d + \iota) BCS(\nabla_s) = BCh(\nabla_1) -  BCh(\nabla_0),
\end{equation}
and this implies $(d\iota + \iota d) BCS(\nabla_s) = 0$.

An explicit formula for $BCS(\nabla_s)$ is given in \cite{TWZ2}. For example, on any local trivialization $V \subset M$, the degree 
$2k+1$ component of this odd differential form on $LV \subset LM$ is given by 

\begin{multline} \label{eq:BCS^V_{2k+1}}
BCS^{V}_{2k+1} (A_s) = Tr\Bigg( \sum_{n\geq k+1}\sum_{\tiny\begin{matrix}{r,j_1,\dots,j_k=1}\\{\textrm{pairwise distinct}}\end{matrix}}^n
\\
 \int_0^1 \int_{\Delta^n}  \iota A_s(t_1)\dots R_s(t_{j_1}  \dots A'_s(t_r)\dots R_s (t_{j_k})\dots \iota A_s(t_n) \quad dt_1\dots dt_n ds 
  \Bigg)
 \end{multline}
 Here $A_s$ is the local expression of the connection with curvature $R_s$ and time derivative $A'_s=\frac{\partial A_s}{\partial s}$ . 
 A similar formula can be obtained, using 
 transition data, for expressing the entire odd form $BCS(\nabla_s)$ on $LM$, see \cite{TWZ2} where the following is also proved.

\begin{prop}  \label{prop:BCSprop}
For any path of connections $\nabla_s$ on $E \to M$, the differential form $BCS(\nabla_s)$ is an element of 
$\Om^{odd}_{d+ \iota}(LM)$
and satisfies
\[
\rho^* BCS(\nabla_s) = CS(\nabla_s),
\]
where $\rho^* :\Om^*_{d+\iota} (LM) \to \Om^*(M)$ is the restriction to constant loops and 
\begin{multline*} \label{eq:CS} 
CS(\nabla_s)\\
= \Tr \left( \int_0^1 \sum_{n\geq 0} \frac{1}{(n+1)!} 
\sum_{i=1}^{n+1} 
\overbrace{ R_s \wedge \dots\wedge R_s\wedge \underbrace{ \left( \nabla_s \right)' }_{i^{\text{th}}}\wedge R_s\wedge\dots\wedge R_s }^{n+1\text{ factors}} \right) ds 
\end{multline*}
is the Chern-Simons form associated to $\nabla_s$.
\end{prop}

Moreover the following multiplicative properties hold, see \cite{TWZ3}.

\begin{thm} \label{prop:BCSsumtensor}
Let $E\to M$ and $\bar E \to M$ be two complex vector bundles, each with a path of connections $(E,\nabla_s) $ and $(\bar E,\bar \nabla_s)$ for $s \in [0,1]$, respectively. 
Let $\nabla_s \oplus \bar \nabla_s$ be the induced path of connections on $E \oplus \bar E$, and let 
$\nabla_r \otimes \bar \nabla_s:= \nabla_r \otimes  Id + Id \otimes \bar \nabla_s$ be the induced connections on $E \otimes \bar E$ for any $r,s\in [0,1]$. Then
\[
BCS(\nabla_s \oplus \bar \nabla_s) = BCS(\nabla_s) + BCS(\bar \nabla_s)
\]
and
\[
BCS(\nabla_0 \ot  \bar \nabla_s) = BCh(\nabla_0) \wedge BCS( \bar \nabla_s) \quad \quad 
BCS(\nabla_s \ot  \bar \nabla_1) = BCS(\nabla_s ) \wedge BCh( \bar  \nabla_1) .
\]
Finallly, for two composable paths of connections $\nabla_s$ and $\bar \nabla_s$, where
$\nabla_1 = \bar \nabla_0$, we have
\[
BCS(\nabla_s  * \bar \nabla_s) = BCS(\nabla_s) + BCS(\bar \nabla_s).
\]
\end{thm}

\begin{rmk} \label{rmk;BCSUniv}
As in Remark \ref{rmk;BChUniv}, we may regard this Bismut-Chern-Simons form $BCS(\nabla_s)$ 
as induced by a map $\phi: M \times I \to BU(k)$, and by the adjoint property there is a universal form 
$BCS \in \Om^{\textrm{odd}}_{d + \iota} (PLBU(k))$ defined by 
\[
BCS = \int_I ev_t^* BCh
\]
where $ev_t : PLBU \times I \to LBU$ is evaluation at time $t \in I$.
To see that this agrees with Definition \ref{defn:BCS}, let $\tilde \phi : M \to PBU$ denote the adjoint mapping
and note that the following diagram commutes
\[
\xymatrix{
\Om^*(LBU ) \ar[r]^-{ev_t^*} \ar[d]^-{(L(\phi))^*} & \Om^*(LP BU   \times I)  \ar[r]^-{\int_I} \ar[d]^-{(L(\tilde \phi) \times id)^*} &  \Om^*(LP BU ) \ar[d]^-{(L(\tilde \phi ))^*} \\
\Om^*(L(M \times I)) \ar[r]^-{j^*}  & \Om^*(LM \times I)  \ar[r]^-{\int_I} &  \Om^*(LM) 
}
\]
\end{rmk}

\begin{rmk}
In \cite{TWZ3} the authors show the Bismut-Chern-Simons form may be used to define a refinement of 
differential $K$-theory. We note that the construction applies to the category of $S^1$-spaces mapping into $L(BU \times \Z)$ with 
universal Bismut Chern form $BCh$, c.f. Remark \ref{rmk;BChUniv}. Then the loop differential
$K$-theory of \cite{TWZ3} is the special case of taking the free loop of a bundle with connection $M \to BU \times \Z$.
\end{rmk}

\section{The odd Chern form} \label{sec;oddCh}

Let $U$ denote the stable unitary group defined as the limit of the finite unitary groups $U(n)$. This is filtered by finite dimensional manifolds and so defines a diffeological space where a finite dimensional plot is a smooth map with image in some $U(n)$.
The group $U$ has a left invariant $1$-form with values in the Lie algebra $\mathfrak u$, which we denote by $\omega$. This generates the canonical bi-invariant closed odd differential form on $U$ given by 
\[
Ch =\Tr \sum_{n\geq 0} \frac{(-1)^n n!}{(2n+1)!}   \omega^{2n+1} \in \Om^{odd}_{cl}(U;\C).
\]
For a smooth map $g: M \to U$, the odd Chern form of $g$ is defined by pullback to be 
\[
Ch(g) = g^*(Ch) =  \Tr \sum_{n\geq 0} \frac{(-1)^n n!}{(2n+1)!}   (g^{-1}dg)^{2n+1} .
\]

We remark that the scalar factor of $1 / (2 \pi i)^{n+1} $ is omitted to be consistent with our convention for the even Chern form.
These factors may be introduced as mentioned above, where for $LM$ below we modifying the complex of forms to have operator $d + (2 \pi i) \iota$. The following lemma explains how the odd Chern form is related to the Chern-Simons form 
of Proposition \ref{prop:BCSprop}. 

\begin{lem} \label{lem;OddChAsCS}
 For any $g \in Map(M, U )$ and any $n$ such that $g(M) \subset U(n)$ we have
\[
Ch(g) = CS(d + s g^{-1} dg),
\] 
where $d + s g^{-1} dg$ is the straight path of connections on the trivial $\C^n$-bundle over $M$, from the trivial connection $d$  to the gauge equivalent flat connection $g^*(d) = d+g^{-1}dg$. 
\end{lem}

The universal version of the statement is as follows: the trivial $\C^\infty$-bundle over 
the stable unitary group comes with a gauge transformation given by left action by $g$ on the fiber over $g$. The 
universal Chern form $Ch$ on $U$ equals the Chern Simons form of the straightline path of connections from the
 trivial connection $d$ to the transform of $d$ by the gauge transformation.

\begin{proof} Let $A = g^{-1}dg$  and $A^s = s g^{-1}dg$ so that $(A^s)'=A$ and $R^s=-s(1-s)A\cdot A$. Then 
\begin{eqnarray*}
CS(d + s g^{-1} dg) &=& \Tr  \sum_{n\geq 0} \frac 1 {(n+1)!}  \sum_{i=1}^{n+1} \int_0^1 R^s \dots \underbrace{(A^s)'}_{i\text{-th}}\dots R^s ds   \\
&=&   \sum_{n\geq 0} \frac 1 {n!}  \Tr \int_0^1 (-s(1-s))^n \underbrace{A\dots A}_{2n+1 \text{ factors}} ds \\
&=& \sum_{n\geq 0} \frac{(-1)^n}{n!} \frac{n!n!}{(2n+1)!}  \Tr(A^{2n+1})  \\
&=& \sum_{n\geq 0} \frac{(-1)^n n!}{(2n+1)!} \Tr(A^{2n+1})  \\ 
&=& Ch(g)
\end{eqnarray*}
where the identity $\int_0^1  s^k (1-s)^\ell ds=\frac{k!\ell!}{(k+\ell+1)!}$ was used.
\end{proof}

It is straightforward the check that  the odd Chern map 
\[
Ch: Map(M, \ U) \to \Omega^{\textrm{odd}}_{cl} (M)
\]
is a monoid homomorphism, i.e.
\[ 
Ch(  g \oplus h ) = Ch(g) + Ch(h),
\]
and furthermore satisfies $Ch(g^{-1}) = -Ch(g)$.  

The odd Chern form $Ch(g)$ depends on $g$, but for a smooth family $g_t : M \times I \to U(n) \subset U$, there is a smooth even differential form that interpolates between $Ch(g_1)$ and $Ch(g_0)$, in the following way. 
We define $CS(g_t) \in \Om^{\textrm{even}}(M)$ by
 \begin{equation}\label{eq:CS(gt)}
 CS(g_t) = \int_I Ch(g_t),
\end{equation}
where $Ch(g_t) \in  \Om^{\textrm{odd}}_{cl}(M\times I)$ and $\int_I$ is the integration along the fiber
\[
\xymatrix{
M \times I \ar[r]^-{g_t} \ar[d]_-{\int_I} & U\\
M & 
}
\]
By Stokes' theorem we have
\begin{equation*}
d CS(g_t) = d\int_I Ch(g_t) = \int_{\partial I} Ch(g_t) - \int_I d Ch(g_t) =Ch(g_1) - Ch(g_0)
\end{equation*}
since $d Ch(g_t) = 0$. By the usual adjunction, our definitions above give a universal form  $CS \in \Om^{even}(PU)$ where $PU$ is the smooth
pathspace of $U$.
 
 The following explicit formula will be useful. For a proof, see for example \cite{TWZ3}.
 \begin{lem} \label{oddCS}
 For any $g_t \in Map(M \times I, U )$ the \emph{odd Chern Simons form} 
 $CS(g_t)  \in  \Omega^{\textrm{even}} (M)$ associated to $g_t$  is  
\[
CS(g_t) = \Tr \sum_{n\geq 0} \frac{(-1)^n n!}{(2n)!} \int_0^1 (g_t^{-1} g_t') \cdot (g_t^{-1} dg_t)^{2n} dt .
 \]
\end{lem}

We restate the fundamental property for $CS(g_t)$ here, along with several others,  whose proofs are immediate from the definitions and Lemma \ref{oddCS}.
\begin{prop} \label{prop:CSprop}
For any paths $g_t, h_t \in Map(M\times I, U )$  we have
\begin{eqnarray*}
  d CS(g_t) &=& Ch(g_1) - Ch(g_0) \\
  CS(g_t \oplus h_t) &=& CS(g_t) + CS (h_t) \\
  CS(g_t^{-1}) &=& - CS(g_t).
\end{eqnarray*}
If $g_t$ and $h_t$ can be composed (\emph{i.e.} if $g_1=h_0$), then the composition $g_t * h_t$ satisfies
\[ CS(g_t*h_t)=CS(g_t)+CS(h_t). \]
Finally, if $g_t : M \to \Omega_*U$, then the degree zero component of $CS(g_t)$ is the integer equal to the winding number.
\end{prop} 

Defining odd $K$-theory by $K^{-1} (M) = [M, U]$, it follows that there is an induced homomorphism 
$Ch : K^{-1}(M) \to H^{\textrm{odd}}(M)$, which is a geometric representation of the odd Chern Character, which is an isomorphism after
tensoring with $\C$.

\section{The odd Bismut-Chern form}

Let $LU$ be the free loop space of $U$.  Motivated by Lemma \ref{lem;OddChAsCS}, we make the following definition.

\begin{defn} \label{defn;BChg}
Let $BCh \in \Om^{odd}(LU)$ be defined by
\[
BCh= BCS(d + s \omega),
\]
where $d + s \omega$ is the path of connections on the trivial $\C^n$-bundle over $U(n)$. 
We refer to this as the universal odd Bismut-Chern form.
\end{defn}

We note that this is well defined since it passes to the limit with respect to $n$.

\begin{thm} \label{BCh:closedrest}
The universal odd Bismut-Chern form is closed, i.e.
\[
(d+ \iota)BCh = 0 ,
\]
so that $BCh \in \Om^{odd}_{d+\iota} (LU)$. 

Also,  $BCh$ is an extension to $LU$ of the universal odd Chern form $Ch\in \Om^{odd}_{cl}(U)$, 
in the sense that 
\[
\rho^*(BCh) = Ch,
\]
where $\rho : U \to LU $ is the inclusion of constant loops. 
\end{thm}

\begin{proof} From Equation \ref{eq:BCSStokes} we have
\[
(d+\iota) BCh = (d + \iota) \left(BCS(d + s \omega) \right) = BCh(d + \omega) - BCh(d) = 0
\]
where the last equality holds since $d = d+ 0$ and $d + \omega$ are gauge equivalent via the left action. (This statement becomes more familiar in  the notation $\omega= g^{-1}dg$, where $g^{-1}dg  = g^{-1} 0 g + g^{-1}dg$.)
For the second statement we have
\[
\rho^* BCS(d + s \omega) = CS(d+ s \omega) = Ch
\]
where the first equality is Proposition \ref{prop:BCSprop} and the second equality is 
Lemma \ref{lem;OddChAsCS}.
\end{proof}

Letting $BCh_{\gamma}$ denote the value of $BCh$ at  $\gamma \in LU$, we see from
Equation \ref{eq:BCS^V_{2k+1}} where
 $A_s = s \omega$, $R_s = -s(1-s) \omega^2$ and $A'_s= \omega$,  
 that an explicit formula for the degree $2k+1$ part of $BCh$ is given by 
\begin{multline} \label{BChexpl}
Tr\Bigg( \sum_{n\geq k+1} \frac{(-1)^k (n-1)! k!}{(n+k)!}   \sum_{\tiny\begin{matrix}{r,j_1,\dots,j_k=1}\\{\textrm{pairwise distinct}}\end{matrix}}^n
 \\
  \int_{\Delta^n}
\iota  \omega (t_1)\dots \omega^2 (t_{j_1})
 \dots \omega (t_r)\dots \omega^2 (t_{j_k})\dots \iota \omega (t_n) \quad dt_1\dots dt_n  
  \Bigg).
 \end{multline}

\begin{lem} \label{lem:BChsum}
Let $BCh_\gamma$ denote the value of $BCh$ at $\gamma \in LU$. Then
the following properties hold
\[
BCh_{\gamma_1 \oplus \gamma_2} = BCh_{\gamma_1}+ BCh_{\gamma_2} \\
\]
and
\[
BCh_{\gamma_1 * \gamma_2} = BCh_{\gamma_1} + BCh_{\gamma_2} 
\]
\end{lem}

\begin{proof}
These are immediate from Definition \ref{defn;BChg} and Theorem \ref{prop:BCSsumtensor}.
\end{proof}

\section{Transgression and odd K-theory}
The form $BCh \in \Om^{odd}_{d + \iota} (LU)$ has a transgression form
$BCS \in \Om^{even}_{d + \iota} (PLU)$ defined by 
\[
BCS = \int_I ev_t^* BCh,
\]
where $ev_t : PLU \times I \to LU$ is evaluation at time $t$.
By the equivariant Stokes formula of section $3$ we have
\[
(d + \iota) BCS = ev_1^* BCh - ev_0^* BCh.
\]
Notice that this equation implies that $BCS \in \Om^{even}_{d+ \iota}(PLU)$.
By a  commutatitive diagram argument similar to Remark \ref{rmk;BCSUniv}
we see that for a map $g_t: M \times I \to U$ we have
$BCS(g_t) \equiv g_t^* BCS$ is given by 
\[
\int_I j^* BCh(g_t),
\]
 where $j : LU \times I \to L(U \times I)$  is inclusion of constant loops on $I$.
 An explicit formula for the degree $2k$ component of $BCS_{\gamma}$  is similar
 to Equation \ref{BChexpl}, with an additional sum over possible $\partial / \partial s$ derivatives of
 any $\omega^2(t_i,s)$ or $\omega(t,s)$:
 \begin{multline}
Tr\Bigg( \sum_{n\geq k+1} \frac{(-1)^k (n-1)! k!}{(n+k)!}   \sum_{\tiny\begin{matrix}{r,j_1,\dots,j_k=1}\\{\textrm{pairwise distinct}}\end{matrix}}^n \\
 \int_I \int_{\Delta^n}
\iota  \omega (t_1)\dots \omega^2 (t_{j_1})
 \dots \frac{\partial}{\partial s} \omega (t_r)\dots \omega^2 (t_{j_k})\dots \iota \omega (t_n) \quad \\
 +   \sum_{1 \leq i \leq k} \int_I \int_{\Delta^n} \pm
\iota  \omega (t_1)\dots \omega^2 (t_{j_1})
 \dots \omega (t_r)\dots  \frac{\partial}{\partial s}  \omega^2 (t_{j_i}) \dots \omega^2 (t_{j_k})\dots \iota \omega (t_n) \quad 
 ds dt_1\dots dt_n  
  \Bigg)
 \end{multline}
 where the sign $\pm$ is positive if $r<j_i$ and negative if $r>j_i$.


 \begin{prop} 
Let $CS \in \Om^{even}(PU)$ denote the universal (even) Chern-Simons form and  
 let $\rho : Map(I, U) \to Map(S^1 \times I, U ) $ be the inclusion of 
 paths in $U$ which are constant in the $S^1$-direction. 
 Then
 \[
 \rho^*(BCS) = CS.
 \]
Consequently, for a map $\gamma: S^1 \times S^1 \to U$, 
 the degree $0$ component of $BCS(\gamma)$ restricts to the winding number of the
 map $S^1 \to U$ given by collapsing the first factor of the torus.
 \end{prop}
 
 \begin{proof} The first statement follows from
  the fact that restriction is compatible with integration along the fiber, as well as Theorem \ref{BCh:closedrest} and 
  Equation \ref{eq:CS(gt)}. 
The second statement follows from the first, in light of the last result in Proposition \ref{prop:CSprop}.
\end{proof}

We have the following corollary, whose even analogue was proved in \cite{Z}, by completely different methods.

\begin{cor} \label{cor:BChfromK} There is a well defined group homomorphism 
\[
[BCh]: K^{-1}(M) \to H^{odd}_{d+\iota}(LM)
\] 
defined by $[g] \mapsto BCh(g)$, for $g:M \to U$, making the following diagram commute
\[
\xymatrix{
 &H^{\textrm{odd}}_{d+\iota}(LM) \ar[d]^{\rho^*} \\
 K^{-1}(M)  \ar [ru]^{[BCh]} \ar[r]^{[Ch]} & H^{odd} (M) 
}
\]
where $[Ch]: K^{-1}(M) \to H^{odd}(M)$ is the odd Chern character,  and $\rho$ is the restriction to constant loops.
\end{cor}

\begin{proof} 
If $g_t$ is a homtopy from $g_0$ to $g_1$ then
\[
(d+ \iota) BCS(g_t) = BCh(g_1) - BCh(g_0),
\]
so $[BCh([g])] \in \Om^{odd}_{d+\iota}(LM)$ is well defined. The map is a group homomorphism by
Lemma \ref{lem:BChsum} and the diagram commutes by Theorem \ref{BCh:closedrest}.
\end{proof}

\end{document}